\documentclass{elsarticle}
\usepackage{graphicx, amsthm, todonotes, amsmath, amssymb,amsfonts, thmtools} 
\usepackage{enumerate}
\usepackage{hyperref}
\usepackage{xcolor}

\usepackage[capitalize,nameinlink]{cleveref}
\usepackage{array}
\usepackage{xcolor}

\newcommand{\youngdiagram}[1]{%
	\begin{tikzpicture}[scale=0.35]
		\foreach \y [count=\i from 0] in {#1} {
			\foreach \x in {1,...,\y} {
				\draw (\x-1,-\i) rectangle (\x,-\i-1);
			}
		}
	\end{tikzpicture}%
}

\newtheorem{thm}{Theorem}
\newtheorem{conj}{Conjecture}
\newtheorem{prop}{Proposition}
\newtheorem{obs}[thm]{Observation}
\newtheorem{cor}[thm]{Corollary}

\newtheorem{lemma}[thm]{Lemma}

\newtheorem*{definition*}{Definition}
\newtheorem{question}[conj]{Question}

\usepackage{thmtools} 
\usepackage{thm-restate}

\crefname{thm}{Theorem}{Theorems} 
\crefname{prop}{Proposition}{Propositions} 
\crefname{conj}{Conjecture}{Conjectures} 
\crefname{obs}{Observation}{Observations} 
\crefname{cor}{Corollary}{Corollarys} 
\crefname{question}{Question}{Questions} 
\crefname{appendix}{Appendix}{Appendices}

\usepackage{stmaryrd}
\newcommand{\br}[1]{\llbracket #1 \rrbracket}
\newcommand{\rs}[1]{\langle #1 \rangle}

\newcommand{\mex}{\operatorname{mex}}
\newcommand{\SG}{{\mathcal G}}
\newcommand{\ZZ}{{\mathbb Z}}
\newcommand{\nim}{\textsc{Nim}}
\newcommand{\rnim}{\textsc{RNim}}
\newcommand{\pnim}{\textsc{PNim}}
\newcommand{\sgrnim}{\SG_\textsc{RNim}}
\newcommand{\sgpnim}{\SG_\textsc{PNim}}
\renewcommand{\P}{\mathcal{P}}
\newcommand{\N}{\mathcal{N}}

\colorlet{darkGreen}{green!50!black}

\newcommand{\class}{
		
	\begin{tikzpicture}[scale=0.5]
		\def\unknown{unknown}
		\def\textsizes{0.6}
		\draw [darkGreen] (2,4.5)--node[fill=none,yshift=1mm,xshift=12mm,scale=\textsizes]{miserable}++(22,0)--++(0,7.5)--++(-22,0)--cycle;
		\draw [red] (3,7.5)--node[fill=none,yshift=1mm,xshift=22mm,scale=\textsizes]{pet}++(16,0)--++(0,3.5)--++(-16,0)--cycle;
		\draw [gray] (3.5,1)--node[fill=none,yshift=1mm,scale=\textsizes]{forced}++(8,0)--++(0,9.5)--++(-8,0)--cycle;
		\draw [cyan] (2.5,0.5)--node[fill=none,yshift=1mm,xshift=-17.5mm,,scale=\textsizes]{returnable}++(17,0)--++(0,11)--++(-17,0)--cycle;
		\draw [blue] (1.5,1.5)--node[fill=none,yshift=1mm,xshift=12mm,scale=\textsizes]{tame}++(23,0)--++(0,11)--++(-23,0)--cycle;

		\node at (15.27,9.5) {$1$-\pnim{}, $1$-\rnim{},};
		\node at (15.27,8.5) {\textsc{Subtraction}};
		\node at (7.5,9) {$1$-\nim{}};
		
		\node at (7.5,6) {\nim{}};
		\node at (15.3,6.5) { \pnim{},\rnim{},};
		\node at (15.3,5.5) {$\textsc{Wythoff}$};
		
		\node at (15.45,3.5){\textsc{Downright},};
		\node at (15.45,2.5){\textsc{LCTR}};
		
		\node at (22,1) {\textsc{Rook}};
	\end{tikzpicture}
}

\newcommand{\str}{
	\begin{tikzpicture}[scale=0.08]%
		\draw (0,0) -- (4,0)--(4,-1)--(3,-1)--(3,-2)--(2,-2)--(2,-3)--(0,-3)--cycle
		(0,-1)--(3,-1)--(3,0)
		(0,-2)--(2,-2)--(2,0)
		(1,0)--(1,-3);%
		\draw[white] (3.4,0)--(3.4,-1);
	\end{tikzpicture}%
}

\usepackage{pgffor,xparse}

\newcommand{\ourtitle}{Nim on Integer Partitions and Hyperrectangles}

\colorlet{myGreen}{green!70!black}
\definecolor{myBlue}{rgb}{0.25, 0.0, 1.0}
\definecolor{lgray}{rgb}{0.75, 0.75, 0.75}

\hypersetup{
	colorlinks=true,
	linkcolor=myBlue,
	citecolor=myGreen,
	urlcolor=myGreen,
	bookmarksopen=true,
	bookmarksnumbered,
	bookmarksopenlevel=2,
	bookmarksdepth=3
	pdftitle = {\ourtitle},
	pdfauthor= {Eric Gottlieb, Matjaz Krnc, Peter Mursic},
}

\begin{document}
\begin{frontmatter}

\title{\ourtitle}
\author[inst1]{Eric Gottlieb}
\ead{gottlieb@rhodes.edu}
\affiliation[inst1]{organization={Rhodes College}, city={Memphis},
            state={Tennessee},
            country={ U.S.A.}}

\author[inst2]{Matjaž Krnc}
\ead{matjaz.krnc@upr.si}
\author[inst2]{Peter Muršič}
\ead{peter.mursic@famnit.upr.si}

\affiliation[inst2]{organization={University of Primorska}, city={Koper},
            country={Slovenia}
            }

\begin{abstract}
 We describe \pnim{} and \rnim{}, two variants of \nim{} in which piles of tokens are replaced with integer partitions or hyperrectangles. 
In \pnim{}, the players choose one of the integer partitions and remove a positive number of rows or a positive number of columns from the Young diagram of that partition. 
In \rnim{}, players choose one of the hyperrectangles and reduce one of its side lengths. 

For \pnim{}, we find a tight upper bound for the Sprague-Grundy values of partitions and characterize partitions with Sprague-Grundy value one. For \rnim{}, we provide a formula for the Sprague-Grundy value of any position. 
We classify both games in the Conway-Gurvich-Ho hierarchy. 
\end{abstract}


\end{frontmatter}

\section{Introduction}
Impartial combinatorial games have been formally studied since the analysis of \nim{} by Bouton \cite{bouton1901} in 1901. 
%
His analysis, together with those of Sprague \cite{Spr35} and Grundy \cite{Gru39}, has led to a rich theory that helps us to determine which positions in a game are losing. Sprague-Grundy values generalize winning and losing positions and are a powerful tool for analyzing disjunctive sums of games. 

Many combinatorial objects lend themselves naturally to game-theoretic interpretations. 
Integer partitions, studied for their number-theoretic and combinatorial properties, are appealing in this setting. 
The Young diagram of a partition offers a geometric, visual medium for encoding discrete structures, and suggests natural move rules: removing parts of a partition corresponds to subtracting, truncating, or reshaping these diagrams in ways that mirror the token-removal actions in classical games like \textsc{Nim}.

\subsection{Related work}

Several researchers have investigated combinatorial games on partitions. 
In 1970, Sato \cite{sato1970maya} showed that Welter's game can be formulated as a game on partitions and conjectured that the Sprague-Grundy values of this game are related to the representation theory of the symmetric group. 
Furthermore, several well-studied games not explicitly played on partitions, including \nim{},  \textsc{Wythoff}, and \textsc{Welter}, can be formulated as games on partitions. 

In 2018, Irie \cite{irie2018p} resolved Sato's conjecture in the affirmative. 
Other researchers, including Abuku and Tada \cite{abuku2023multiple} and Motegi \cite{motegi2021gamepositionsmultiplehook}, further extended these results. 
In addition, a number of other 
games on integer partitions have been studied.
Several authors \cite{Gottlieb2024LCTR,Ilic2019} studied the game \textsc{LCTR}. 
Ba\v si\' c \cite{basic2022some,Basic_2023} studied the game \textsc{CRIM}. 
A suite of games motivated by the moves of chess pieces, collectively referred to as \textsc{Impartial Chess}, was studied by Berlekamp \cite{impchess} and others \cite{gottlieb2025impartialchessintegerpartitions}. 
All of these games are impartial. In her honors thesis, Meit \cite{hannah2025thesis} studied \textsc{CRPM} and \textsc{CRPS}, two partizan combinatorial games on partitions.


\subsection{Our results}

We introduce \pnim{}, an impartial combinatorial game in which a position consists of several Young diagrams of integer partitions. 
Players alternate choosing a nonempty partition and removing either a positive number of rows or a positive number of columns from its Young diagram. If the rows (or columns) removed lie between rows (or columns) that are not removed, then the remaining rows (or columns) are merged. 
For example, if the last two rows of the first partition in $\br{4, 2, 1}+\br{3, 3}$ are removed, the resulting partition is $\br{4}+ \br{3,3}$. 
\nim{} is the special case of \pnim{} where each partition consists of a single part.

We study Sprague-Grundy values for \pnim{} when played on a single partition.
We determine those values for various families, including the family of rectangular partitions (i.e. rectangles).
\begin{restatable}{thm}{thmrect}
\label{thm:rect}
    If $r$ and $c$ are positive integers then \[\sgpnim(\br{c^r}) = ((r-1) \oplus (c-1)) + 1.\]
\end{restatable}
We establish a tight upper bound for Sprague-Grundy values for \pnim{}.
\begin{restatable}{prop}{thmbound}
\label{prop:bound}
    If $\lambda = \br{\lambda_1, \ldots, \lambda_r}$ is nonempty then $\sgpnim(\lambda) \leq \lambda_1 + r - 1$. 
\end{restatable}
We identify several infinite families of partitions which attain the upper bound from \cref{prop:bound}. At the other extreme, we characterize the partitions with Sprague-Grundy value $1$.
\begin{restatable}{thm}{thmpartitioninterval}
\label{thm:partition_interval}
For any partition $\lambda$ we have $\sgpnim(\lambda) = 1$ if and only if $\lambda=\br{1}$ or $\br{r,r,r-1,r-2,\ldots,2}\leq \lambda \leq \br{r^r}$ for some $r\geq 2$.
\end{restatable} 
We show how \cref{thm:partition_interval} can be used to find the optimal response for \pnim{} under mis\`ere play when played on a single partition.
Motivated by \cref{thm:rect} we consider a related game, denoted by \rnim{}, which is played on hyperrectangles, which are represented by $\rs{k_1, \dots, k_d}$, where each $k_j$ is a nonnegative integer. 
A position consists of several hyperrectangles, possibly of varying dimension.
A move consists of choosing a hyperrectangle of positive hypervolume and reducing the length of one of its sides. 
For example, 
both $\rs{5,1,2}+\rs{2,3}$ and $\rs{5,4,2}+\rs{0,3}$ 
can be reached from  $\rs{5,4,2}+\rs{2,3}$ 
in a single move.
In this setting \cref{thm:rect} generalizes to any dimension $d$.
\begin{restatable}{thm}{thmformula}
\label{thm:formula}
    If $k_1, \ldots, k_d$ are positive integers, then \[ \sgrnim(\rs{k_1, \ldots, k_d}) = ((k_1-1) \oplus \cdots \oplus (k_d-1)) + 1.\]
\end{restatable}

We also study mis\`ere variants of both \pnim{} and \rnim{}. 
We establish that both games, when played on a single partition or hyperrectangle, are pet and returnable. 
Due to \cref{thm:formula,thm:partition_interval,thm:tame}, for \pnim{} played on a single partition, or \rnim{} played on several hyperrectangles, we are able to respond optimally, under both normal as well as mis\`ere play.  
We also show that resolving \pnim{} under normal play is equivalent to resolving it under mis\`ere play.

This paper is structured as follows. In \cref{sec:prelim} we provide the necessary definitions and conventions that we use. 
\cref{sec:pnim} is dedicated to our results on \pnim{} including \cref{thm:rect,prop:bound,thm:partition_interval}, followed by \cref{sec:rnim} which concerns \cref{thm:formula} and other results on \rnim{}. 
We classify \pnim{} and \rnim{} in the sense of Conway-Gurvich-Ho in \cref{sec:cgh}, where we also explain how to play misère.
\cref{sec:conclusion} summarizes connections with other areas of mathematics and offers directions for future work. 
 \cref{sec:small,sec:heavy} contain data in support of the conjectures and questions posed in \cref{sec:conclusion}.


\renewcommand{\sgrnim}{\SG_\textsc{R}}
\renewcommand{\sgpnim}{\SG_\textsc{P}}
\section{Preliminaries} \label{sec:prelim}

We denote the set of integers by $\mathbb Z$ and the set of nonnegative integers by $\mathbb Z^{\geq 0}$. For integers $i$ and $j$ we define $[i,j] = \{ n \in \mathbb Z : i \leq n \leq j\}$ and $[j]=[1,j]$. 

\paragraph{Partitions} Let $n \in \mathbb Z^{\geq 0}$. 
An \emph{(integer) partition} $\lambda$ of $n$ is a sum $\lambda_1 + \cdots + \lambda_r = n$ of positive integers equal to $n$ with $\lambda_1 \geq \cdots \geq \lambda_r$. The $\lambda_i$'s are called \emph{parts}. 
We write $\lambda \vdash n$, and $\lambda = \br{\lambda_1, \ldots, \lambda_r}$, and use the exponent notation to shorten the repeated parts of same size, e.g., $\br{3,3,2,2,2}=\br{5^0,3^2,2^3}=\br{3^2,2^3}$.
The empty partition $\br{ \, }$ is the unique partition of zero. 
Let $\mathbb Y = \{\lambda \vdash n : n \in \mathbb Z^{\geq 0}\}$. 
The \emph{(Dyson's) rank} of $\lambda$ is defined to be $|\lambda_1-r|$; see \cite{dyson1944some}.

A \emph{Young diagram} is a way to represent a partition graphically. Since partitions are in bijection with  Young diagrams, we do not distinguish between a partition and its Young diagram.
The \emph{conjugate} of a partition is obtained by exchanging the rows and columns of its Young diagram; see Andrews \cite{andrews1998theory}. 

Young's lattice is a partial order on $\mathbb Y$ defined by $\br{\lambda_1, \ldots, \lambda_\ell} \leq \br{\mu_1, \ldots, \mu_m}$ if and only if $\ell \leq m$ and $\lambda_j \leq \mu_j$ for each $j = 1, \ldots, \ell$. This partial order is relevant for us in \cref{thm:partition_interval}, for example. 

\paragraph{Impartial games}  
All the combinatorial games in this paper are finite and impartial;  \cite{Sie13,Con76,berlekamp2018winning} are standard references. 
A game consists of rules that dictate the moves the players can make from a given position.
We write $p \to p'$ if there is a move from position $p$ to position $p'$. 
When clear from context, we use terms \emph{position} and \emph{game} interchangeably. 
%
A position is said to be \emph{terminal} if no moves are available from it. Under normal (resp. misère) play, moving to a terminal position wins (resp. loses) the game. 

In an impartial game, every position has the property that exactly one of the two players can win if they play optimally. 
If the next player can force a win, the position is called \emph{winning} or an \emph{$\N$-position}. 
If the previous player can force a win, the position is called \emph{losing} or a \emph{$\P$-position}. 
We define $\N$ and $\P$ to be the sets of winning and losing positions, respectively. 

For a finite set $S \subseteq \ZZ^{\leq 0}$, define $\mex(S)$ to be the smallest nonnegative integer not contained in $S$. 
The \emph{Sprague-Grundy} value of a position $p$ in a game $G$ is $\SG_G(p) = \mex \{\SG_G(p') \mid p \to p'\}$. 
In particular, if $p$ is terminal, we have $\SG_G(p)=0$. 
The \emph{misère Grundy} value $\SG_G^-(p)$ is defined the same way as $\SG_G(p)$, except terminals are assigned value $1$. 
Under normal (resp. misére) play, a position is  losing if and only if its Sprague-Grundy (resp. misére Grundy) value is zero. 
A natural upper bound for both Sprague-Grundy as well as misère Grundy value is the maximal possible number of moves from $p$ to a terminal position of $G$ (i.e. \emph{the longest play} from $p$).
A \emph{$(t,\ell)$-position} is any position  $p$ with $\SG_G(p)=t$ and $\SG_G^-(p)=\ell$, while a \emph{$t$-position} is any position  $p$ with $\SG_G(p)=t$.

The \emph{disjunctive sum of two games} (or game sum) $G_1+G_2$ is a game in which the two games are played in parallel, with each player being allowed to make a move in either one of the games per turn.
The disjunctive sum is commutative and associative.
By Sprague-Grundy theorem, the $\SG_{G_1+G_2}$-value of position $(\lambda_1,\lambda_2)$ in the game sum $G_1+G_2$ equals the Nim-sum of Sprague-Grundy values of the games, that is
$\SG_{G_1}(\lambda_1) \oplus \SG_{G_2}(\lambda_2)$, where the $\oplus$ operator denotes the binary \textsc{xor} operation (also called \emph{nimber addition}).


\section{\pnim{}} \label{sec:pnim}
This section studies Sprague-Grundy values for \pnim{} under normal play convention. We focus on the restriction of \pnim{} to positions consisting of a single partition, which we call \emph{$1$-\pnim{}}. 
We establish bounds for Sprague-Grundy values, characterize partitions attaining value $1$, and provide complete solutions for special partition families including rectangles. 
The properties we uncover for $1$-\pnim{} will later be used towards analyzing both normal and misère play of \pnim{} in general.

In \pnim{}, players alternate selecting one of the partitions and removing a positive number of rows or a positive number of columns. 
If there remains more than one piece of the diagram initially chosen, those pieces are merged together. We now define this formally.

\begin{definition*}
\pnim{} is a disjunctive sum of several $1$-\pnim{} games.
\mbox{$1$-\pnim{}} is a game in which positions are members of $\mathbb Y$. 
The empty partition $\br{\,}$ is the only terminal position in $1$-\pnim{}. 
Let 
$\lambda=\br{\lambda_1,\dots,\lambda_r}$ be such a member and let $\lambda'=\br{\lambda'_1,\dots,\lambda'_{\lambda_1}}$ be its conjugate. 
In $1$-\pnim{}, we have row moves (1) and column moves (2). 

\begin{enumerate}
    \item For a proper subsequence $i_1,\dots, i_\ell$ of $1,\dots,r$, we have $\lambda\to\br{\lambda_{i_1},\dots,\lambda_{i_\ell}}$.

    \item For a proper subsequence $i_1,\dots, i_\ell$ of $1,\dots,\lambda_1$, we have $\lambda\to \br{\lambda'_{i_1},\dots,\lambda'_{i_\ell}}'$. 
\end{enumerate}
We denote the Sprague-Grundy value of a position in \pnim{} by $\sgpnim(\cdot)$. 
\end{definition*}
\noindent An example of a \pnim{} play starting from $\br{4,2,1} + \br{3,3}$ is shown below.
    \begin{equation*}
    \label{eq:pnimeg}
    \fcolorbox{white}{black!20}{$\br{4,2,1}\!+\!\br{3,3}$}
    \to 
    \fcolorbox{white}{black!20}{$\br{4}\!+\!\br{3,3} $}
    \to
     \fcolorbox{white}{black!20}{$\br{4}\!+\!\br{3} $}
    \to 
    \fcolorbox{white}{black!20}{$\br{3}\!+\!\br{3}$}
    \to
    \fcolorbox{white}{black!20}{$\br{3}$}
    \to
    \fcolorbox{white}{black!20}{$\br{ \, }$}
    \end{equation*}

\begin{obs}[conjugate invariance]\label{obs:conjinvariance}
Let $\lambda$ be a partition and $\lambda'$ be its conjugate. Then $\sgpnim(\lambda)=\sgpnim(\lambda')$, as well as $\sgpnim^-(\lambda)=\sgpnim^-(\lambda')$. 
\end{obs}

The following result is perhaps surprising since the number of moves on $\lambda$ would seem to grow exponentially in the largest part and number of parts of $\lambda$. We denote the Sprague-Grundy value of a position in \pnim{} by $\sgpnim$. 


\begin{prop}
    The longest \pnim{} play from $\br{\lambda_1,\dots,\lambda_r}\neq \br{ \, }$ is of length $\lambda_1+r-1$.
\end{prop}
\begin{proof}
For any partition $\mu$ we define $f(\mu)$ to be the cumulative count of its rows and columns. 
It is easy to see that  the longest \pnim{} play from $\lambda=\br{\lambda_1,\dots,\lambda_r}$ does not exceed $f(\lambda)-1$. 
This follows from the fact that for every $\mu$ such that $\lambda \to \mu$ we have $f(\lambda)-f(\mu)\ge 1$, while for the last move we have $f(\lambda)-f(\br{ \, })\ge 2$.

To conclude the proof observe that the play of length $f(\lambda)-1$ can always be realised by iteratively removing the 2nd row, or the 2nd column, as long a possible, reaching a partition $\br{1}$ after $\lambda_1+r-2$ moves.
\end{proof}

Since the longest play is a trivial upper bound for Sprague-Grundy value, \cref{prop:bound} follows directly. 
We say that a position is \emph{heavy} if its Sprague-Grundy value is the length of the longest play. It turns out that, under $1$-\pnim{}, many families of partitions are in fact heavy.
This motivates \cref{prop:hook,prop:padded_stair,obs:lucas,conj:heavychoppedrect,conj:shallowstair,conj:heavy}.

\begin{prop}\label{prop:hook}
    Let $r$ and $c$ be positive integers. 
    Then  $\br{c, 1^{r-1}}$ is heavy.
\end{prop}
\begin{proof}
We prove the claim by induction on $r+c$.
     Since  $\sgpnim(\br{1}]) = 1$, we assume $r+c>2$. 
    For any $k < r + c - 1$, we claim that there is a move $\lambda \to \mu$ such that $\mu \vdash k$. Without loss of generality suppose $r \leq c$. If $k < c$ then remove the first $c-k$ columns of $\br{c, 1^{r-1}}$ to obtain  $\mu = \br{k}$. If $k \geq c$ remove the last $r+c-k-1$ rows of $\br{c, 1^{r-1}}$ to obtain $\mu = \br{c, 1^{k-c}}$. In either case, $\sgpnim(\mu) = k$. 
    
    There is a move giving a position with Sprague-Grundy value equal to any value less than $r + c - 1$, so $\sgpnim(\br{c, 1^{r-1}}) \geq r + c - 1$. By \cref{prop:bound}, we have $\sgpnim(\br{c, 1^{r-1}})  \leq r + c - 1.$ Thus, equality holds. 
\end{proof}

\noindent For $c\ge r \ge 1$ let $\str_{c,r}$ denote the partition $\br{c,c-1,\dots,c-r+1}$.

\begin{prop}
\label{prop:padded_stair}    For $c\ge r \ge 1$, the partition $\str_{c,r}$ is heavy.
\end{prop}
\begin{proof}
    By \cref{prop:bound} it is enough to prove that from $\str_{c,r}$ we can reach a \mbox{$j$-position} for any $j\in [0,r-1]\cup[r, 2r-2]\cup[2r-1, c+r-2]$.
    We prove the claim by induction on $r$.
    For the base case $r=1$ the result follows from \cref{prop:hook}.
    For  the induction step assume that $r>1$. We show the existence of a move to a $j$-position in three steps. In particular,
    \begin{description}
        \item[Case {$j\in[0,r-1]$:}] we obtain a $j$-position for any $j \in [0,r-1]$ by removing everything except a single column of length $j$ if $j > 0$, or by removing everything if $j = 0$.
        \item[Case {$j\in [r, 2r-2]$:}] we obtain a $(2r-i-1)$-position for any $i\in [r-1]$ by removing all but $1$ column of length $r$ and all columns of length at most $i$, yielding a partition conjugate to $\str_{r,r-i}$, which has the same $\sgpnim$ value, by \cref{obs:conjinvariance}.
        \item[Case {$j\in[2r-1, c+r-2]$:}] we obtain a $((c-i)+r-1)$-position $\str_{c-i,r}$ for any $i\in [c-r]$ by removing $i$ columns of length $r$ by induction. This gives all values of $j$ between $c+r-1-(c-r)=2r-1$ and $c+r-2$ inclusive. \qedhere
    \end{description}
\end{proof}

We now turn our attention to rectangular partitions. In order to determine which rectangles are heavy, we first determine their  
Sprague-Grundy values, which in turn reveal a nondisjunctive connection to nimber arithmetic.
\thmrect*
We proceed with two technical lemmas.
Although both lemmas could be incorporated directly into the proof of \cref{thm:rect}, we state them separately for future reference in \cref{sec:rnim}.

\begin{lemma}\label{lem:shifted_mex}
Let $k\geq 1$ and $S\subseteq \mathbb Z_{\geq 0}$. Then
\[\mex(S)+k=\mex(\{0,\ldots,k-1\}\cup \{s+k:s\in S\}).\]
\end{lemma}
\begin{proof}\belowdisplayskip=-12pt
The claim immediately follows due to
 \begin{align*}
 k+\mex(S)&=k+\min(\mathbb Z_{\geq0} \setminus S)\\
 &=\min(\mathbb Z_{\geq k} \setminus \{s+k:s\in S\})\\
 &=\min(\mathbb Z_{\geq 0} \setminus (\{0,\ldots,k-1\} \cup  \{s+k:s\in S\}))\\
 &=\mex(\{0,\ldots,k-1\}\cup \{s+k:s\in S\}).
 \end{align*}
\end{proof}

We only need case $k=1$ of the next lemma to prove \cref{thm:rect}, but we present the general form as the proof requires no additional effort.

\begin{lemma}\label{lem:nimformula}
Let $r,c,k$ be positive integers with $r,c\geq k$. Then 
\[
((c-k)\oplus (r-k))+k=\mex\!
\begin{pmatrix}
    0,\ldots,k-1,\\
    ((r-k)\oplus 0)+k,\ldots,((r-k) \oplus (c-k-1))+k,\\
    (0 \oplus (c-k))+k,\ldots,((r-k-1)\oplus (c-k))+k)
\end{pmatrix}.
\]  
\end{lemma}

\begin{proof} By definitions of Nim-sum and mex, and for integers $r,c\geq0$ we have 
\begin{align*}
     r \oplus c &=\mex \!{\bigl (}\{r '\oplus c :0 \leq r '<r \}\cup \{r \oplus c ':0 \leq c '<c \}{\bigr )} 
\end{align*}
\noindent For $r,c\geq k >0$, substituting $r - k$ for $r$ and $c-k$ for $c$ yields $(r -k) \oplus (c -k)=$
 \[ \mex \bigl (   \{r '\oplus (c-k) :0 \leq r '<r-k \}\cup \{(r-k) \oplus c ':0 \leq c '<c-k \} \bigr ). \]

By \cref{lem:shifted_mex} we obtain 
\[
((r -k) \oplus (c -k)) +k 
=
\mex 
\begin{pmatrix}
    [0,k-1]   \cup  \{(r '\oplus (c-k))+k :0 \leq r '<r-k \} \\
     \phantom{[0,k-1]} \cup \{((r-k) \oplus c')+k:0 \leq c '<c-k \} 
\end{pmatrix}, \]
as desired. \end{proof} 

\noindent We are now in position to prove \cref{thm:rect}.  

\begin{proof}[Proof of \cref{thm:rect}]
We use induction on $r+c$. It holds when $r+c=2$. For $r+c>2$ assume that the theorem holds for all rectangles smaller than $\br{c^r}$. The set of moves from $\br{c^r}$ is 
$\{\br{ \, }\} \cup \{\br{a^r} : 1\leq a <c\} \cup \{\br{c^b} : 1 \leq b <r\}$ so
\begin{align*}
\sgpnim(\br{c^r})&=\mex
\begin{pmatrix}
0,&((r-1)\oplus 0)+1,\ldots,((r-1)\oplus (c-2))+1,\\
&(0 \oplus (c-1))+1,\ldots, ((r-2) \oplus (c-1))+1
\end{pmatrix}\\
&=((r-1)\oplus(c-1))+1,
\end{align*}
where the last equality follows from \cref{lem:nimformula} for $k=1$.
\end{proof}
Which rectangles are heavy? The answer is a consequence of Lucas' number-theoretic result. 
\begin{thm}[Lucas \cite{lucas}]
    Let $m$ and $n$ be non-negative integers, let $p$ be a prime, and let 
\begin{align*}
    m&=m_kp^k+m_{k-1}p^{k-1}+\cdots +m_1p+m_0 \quad \text{ and}\\
    n&=n_kp^k+n_{k-1}p^{k-1}+\cdots +n_1p+n_0
\end{align*}
be the base-$p$ expansions of $m$ and $n$ respectively. Then
\[\binom{m}{n}\equiv\prod_{i=0}^k\binom{m_i}{n_i}\pmod p.\]
\end{thm}
\noindent The following is a well-known consequence of Lucas' theorem \cite{rowland2011numbernonzerobinomialcoefficients}. 
\begin{cor}\label{cl:lucas}
    The value $\binom{m}{n}$ is odd if and only if the position of $1$-digits in the binary representation of  $n$ are a subset of those in $m$.
\end{cor}

We are now ready to characterize heavy rectangles, i.e., those for which we can replace the Nim-sum in \cref{thm:rect} with ordinary addition.  
\begin{obs}\label{obs:lucas}
    The rectangle $\br{c^r}$ is heavy if and only if $\binom{c + r -2}{r-1}$ is odd.  
\end{obs}
\begin{proof}

\noindent Substituting $m$ and $n$ with $c + r -2$ and $r-1$ in \cref{cl:lucas} yields 
\[
\binom{c + r -2}{r-1}\quad\text{odd} \iff (c-1)\oplus(r-1)=(c-1) + (r-1), 
\]
from where our claim follows due to \cref{thm:rect}.
\end{proof}

The next result characterizes heaviness for another family of partitions, namely those with two parts. 

\begin{prop}\label{prop:2rows}
     Let $c_1$ and $c_2$ be integers with $c_1 \geq c_2 > 0$. Then 
    \[\sgpnim(\br{c_1, c_2}) = \begin{cases}
        c_1 - 1 & \textnormal{if }c_1 = c_2 \textnormal{ are even} \\
        c_1 + 1 & \textnormal{otherwise} 
    \end{cases}.\]
\end{prop}

\begin{proof}
    The first case follows from \cref{thm:rect} with $r = 2$ as we get \[((2-1) \oplus (c_1-1))+1=(c_1-2)+1=c_1-1.\]
    If $c_1=c_2$ are odd, we get $((c_1-1) \oplus (2-1))+1=c_1+1$.
%
    If $c_1>c_2$, let $R$ and $S$ be defined by
    \begin{align*}
       R&=\{\br{c_1}, \br{c_2}\} \cup \{\br{a, b} : a \geq b \geq 0, a-b \leq c_1 - c_2, a < c_1, b \leq c_2\} \mbox{  and}\\
       S&= \{\br{i}:i \in [0,c_1-c_2]\} \cup \{\br{c_1-c_2,1}\}\cup\{\br{c_1-c_2+i,i}:i \in [1,c_2-1]\}.
   \end{align*}
    Then $R$ is the set of positions reachable from $\br{c_1, c_2}$ in one move and $S \subseteq R$, so by induction we have 
\begin{align*}
\sgpnim(\br{c_1, c_2})&=\mex(\{\sgpnim(\lambda): \lambda \in R\})\\
&\geq \mex(\{\sgpnim(\lambda): \lambda \in S\})\\
&=\mex([0,c_1-c_2]\cup\{c_1-c_2+1\}\cup [c_1-c_2+1,c_1])\\
&=\mex([0,c_1])= c_1 + 1.
\end{align*}

To see that equality holds, we need only show that no move from $\br{c_1, c_2}$ has Sprague-Grundy value $c_1 + 1$. If the move consists of removing one or both rows, then the resulting positions have Sprague-Grundy values 0, $c_1$, or $c_2$, all of which are less than $c_1 + 1$. If the move consists of removing columns, then the resulting partition has fewer columns, say $\bar c_1$ of them, so its Sprague-Grundy value is at most $\bar c_1+1$ by induction which is strictly less then $c_1+1$.
\end{proof}




We now turn our attention to the partitions which have Sprague-Grundy value $1$. As we will see later, those positions are important as they correspond exactly to the losing positions of $1$-\pnim{} under misère play.

\thmpartitioninterval*
\begin{proof}
Let
\[S=\{\br{1}\} \cup \bigcup_{r \geq 2} \{\lambda \in \mathbb Y: \br{r,r,r-1,r-2,\ldots,2}\leq \lambda \leq \br{r^r}\}.\] 
    We show that (i) every partition not in $S \cup \{\br{ \, }\}$ has a move to a partition in $S$ and (ii) there is no move from a partition in $S$ to a partition in $S$. 
\begin{enumerate}[(i)]
    \item Let $\br{\lambda_1,\dots,\lambda_r}\notin S \cup \{\br{ \, }\}$. Without loss of generality assume $\lambda_1\ge r$. 
    In this proof, set $\lambda_i=0$ for $i\in [r+1,\lambda_1]$. If $\lambda_1> \lambda_2$, then removing the first $\lambda_1-1$ columns yields $\br{1}$.
    Notice that for $2\le i\le \lambda_1$ we have 
    \[\br{\lambda_1,\dots,\lambda_r}\in S \iff \lambda_i\ge \lambda_1-i+2.\]
    Thus, there is a smallest $i > 1$ such that  $\lambda_i < \lambda_1-i+2$. Removing the first $\lambda_1-i+1$ columns gives a partition in $S$. 

    \item Note that all partitions in $S$ have rank $0$, so it is enough show that any move from a partition in $S$, other than to $\br{ \, }$, results in a nonzero rank.
    Furthermore, since $S$ is closed under conjugation, we may assume that we are removing rows.
    Indeed, removing 
    $k<r$ rows from $\lambda$ will result in a partition $\bar\lambda\geq \br{r-k+1,r-k,\ldots,2}$, thus $\bar\lambda$ has $r-k$ rows and at least $r-k+1$ columns, hence $\bar\lambda$ does not have rank $0$. \qedhere
    
\end{enumerate}    
\end{proof}

\section{\rnim{}}\label{sec:rnim}

\pnim{}, when played on  rectangles, can be viewed as a game played on several pairs $\langle r, c \rangle$ of nonnegative integers. Players alternate selecting a pair in which both entries are positive and reducing one of the entries of the chosen pair by a positive integer. The terminal positions are those positions in which at least one of the coordinates of each pair in the sequence is zero. 
\pnim{} on rectangles can be generalized to a game \rnim{} played on several hyperrectangles as follows. 

\begin{definition*}
\rnim{} is a disjunctive sum of several $1$-\rnim{} games. 
\mbox{$1$-\rnim{}} is a game in which positions are hyperrectangles. We represent a $d$-dimensional hyperrectangle as the $d$-tuple $\rs{k_1, \dots, k_d}$ of its side lengths, where each $k_j$ is a nonnegative integer. 
The terminal positions are those with a side of length zero. 
In this setting, $\rs{k_1, \dots, k_d} \to \rs{\ell_1, \dots, \ell_d}$ if there exists $i \in [d]$ such that $0 \leq \ell_i < k_i$ and $\ell_j = k_j$ for $j \neq i$. 
We denote the Sprague-Grundy value of a position in \rnim{} by $\sgrnim(\cdot)$. 
\end{definition*}

For example, play on the initial position consisting of the $d$-tuples $\rs{3,5}$ and $\rs{3,3,3}$ may proceed as shown below.
    \[
     \fcolorbox{white}{black!20}{$\rs{2,1}+\rs{3,3,3}$}
    \to
    \fcolorbox{white}{black!20}{$\rs{2,1}+\rs{3,3,2} $}
    \to
     \fcolorbox{white}{black!20}{$\rs{0,1}+  \rs{3,3,2} $}
     \to
    \fcolorbox{white}{black!20}{$\rs{0,1}+ \rs{3,0,2}$}
    \]    

In \cref{thm:formula} we  determine Sprague-Grundy values for $1$-\rnim{}, which in turn provides Sprague-Grundy values for \rnim{}.
We first establish the following supporting results.

\begin{obs}\label{obs:mex-xor}
    For nonnegative integers $k_1,\ldots,k_d$ we have
    \[\bigoplus_{i=1}^d k_i =\mex  \left(\bigcup_{i=1}^d \left\{ \left(\bigoplus_{j\neq i} k_j  \right)\oplus k_i':0 \leq k_i' <k_i \right\}\right) \]
\end{obs}

\begin{proof}
To prove the claim it is enough to observe that the left-hand side is equal to the Sprague-Grundy value of the \nim{} position with piles of size $k_1, \ldots, k_d$ (see \cite{bouton1901}), while the members of the set under $\mex$ on the right-hand side correspond exactly to all positions reachable from that \nim{} position.
\end{proof}

The next statement is an immediate corollary of \cref{lem:shifted_mex,obs:mex-xor}. The proof, which we omit, is similar to that of \cref{lem:nimformula}. 
\begin{cor}\label{cor:rnim}
    For integers $k_1,\ldots, k_d>k\geq 0$ we have $k+\bigoplus_{i=1}^d (k_i-k) = $
    \[\mex \left( [0,k-1] \cup \Bigg(\bigcup_{i=1}^d \{ k+\bigg( k_i' \oplus \bigoplus_{j\neq i} (k_j-k)  \bigg):0 \leq k_i' <k_i-k \}\Bigg) \right) .\]
\end{cor} 

\noindent The following result generalizes \cref{thm:rect}. 
\thmformula*
\begin{proof}
    By contradiction assume that the claim is false and let
$p=\rs{k_1,\dots,k_{d}}$ be the lexicographically smallest $d$-tuple violating the claim. The set of moves from $\lambda$ is equal to $S_0\cup S_1$, where 
\begin{align*}
  S_0&=\{(k_1,\ldots,k_{i-1},0,k_{i+1},\ldots,k_d) : 1 \leq i \leq d\} \text{ \ \ \ and}\\
S_1&=\{(k_1,\ldots,k_{i-1},k'_i,k_{i+1},\ldots,k_d) : 1 \leq i \leq d, \ 1\leq k_i'<k_i\}.
\end{align*}

\noindent
Using this notation, we can write
\begin{align*}
    \sgrnim(p)&=\mex(\sgrnim(\overline p) :\overline  p \in S_0 \cup S_1) \\
    &=\mex(\{\sgrnim(\overline p) :\overline{p} \in S_1\}\cup \{0\}) \\
    &=\mex  \left( \{0\} \cup \left(\bigcup_{i=1}^d \bigg\{ 1+\bigg( k_i' \oplus \bigoplus_{j\neq i} (k_j-1)  \bigg):0 \leq k_i' <k_i-1 \bigg\}\right)\right) \\
    &=1+\bigoplus_{i=1}^d (k_i-1),
\end{align*}
where the second equality follows due to positions in $S_0$ being terminal, the third equality follows due to the minimality, and the last equality follows due to \cref{cor:rnim} with $k=1$.
\end{proof}

\cref{thm:formula} gives the Sprague-Grundy value of any position of \rnim{} under normal play via nimber addition.

\begin{cor}\label{cor:rnimformula}
    For $i\in [\ell]$, let $p_i=\rs{k^i_1, \ldots, k^i_{d_i}}$ be a position in $1$-\rnim{}. 
   Then $\sgrnim(p_1 + \cdots + p_\ell)=g_1\oplus \dots \oplus g_\ell$, where 
    \[
    g_i=\begin{cases}
       0 & \text{if }\Pi_{j=1}^{d_i}k^i_j=0;\\
       1+\bigoplus_{j=1}^{d_i} (k^i_j-1) & \text{otherwise.}
    \end{cases}
    \]
\end{cor}

\section{Conway-Gurvich-Ho classification and misère} \label{sec:cgh}

The interaction of the normal and mis\`ere forms of combinatorial games was initiated by Conway \cite{Con76} and expanded by Gurvich and Ho \cite{GURVICH201854}. The so-called \emph{CGH classification} is defined by certain properties which we now recall.
The domestic class is not relevant for this paper so we omit it.
\begin{definition*}[CGH classification] \label{D.DTP}
An impartial game is called \begin{enumerate} 
\item {\em returnable} if for any move from a $(0,1)$-position (resp.,~a $(1,0)$-position) to a non-terminal position $y$, there is a move from $y$ to a $(0,1)$-position (resp.,~to a $(1,0)$-position);
\item {\em forced} if each move from a $(0,1)$-position results in a $(1,0)$-position and vice versa;
\item  {\em miserable} if for every position $x$, one of the following holds: (i) $x$ is a $(0, 1)$-position or $(1, 0)$-position, or (ii) there is no move from $x$ to a $(0,1)$-position or $(1, 0)$-position, or (iii) there are moves from $x$ to both a $(0, 1)$-position and a $(1, 0)$-position;
\item {\em pet} if it has only $(0,1)$-positions, $(1,0)$-positions, and $(k,k)$-positions with $k \geq 2$;
\item {\em tame} if it has only $(0,1)$-positions, $(1,0)$-positions, and $(k,k)$-positions with $k \geq 0$.
\end{enumerate}
\end{definition*}
According to the above definitions, tame, pet, miserable, returnable and forced games form nested classes:
a pet game is returnable and miserable, a miserable game is tame, 
and a forced game is returnable; see \cref{fig:CGclass}. 

\newpage
\begin{prop}\label{prop:pet}
    The games $1$-\pnim{} and $1$-\rnim{} are pet and returnable, but not forced.
\end{prop}
\begin{proof}
Gurvich and Ho proved that a game is pet if and only if it does not admit a $(0,0)$-position; see
\cite[Theorem 5]{GURVICH201854}.
    We show that there are no $(0,0)$-positions in $1$-\pnim{} and $1$-\rnim{}. 
    Every nonterminal position has a move to a terminal position, so only terminal positions are $0$-positions. By definition the terminal positions are $(0,1)$-positions, implying that there are no $(0,0)$-positions. Furthermore, pet games are returnable; see \cite[Proposition 1]{GURVICH201854}. 
    
    The position $\br{2,2}=\rs{2,2}$ is a $(1,0)$-position in both games. It has a move to a $(2,2)$-position $\br{2}=\rs{2,1}$, so 1-\pnim{} and 1-\rnim{} are not forced.
\end{proof}

\begin{figure}[!h]
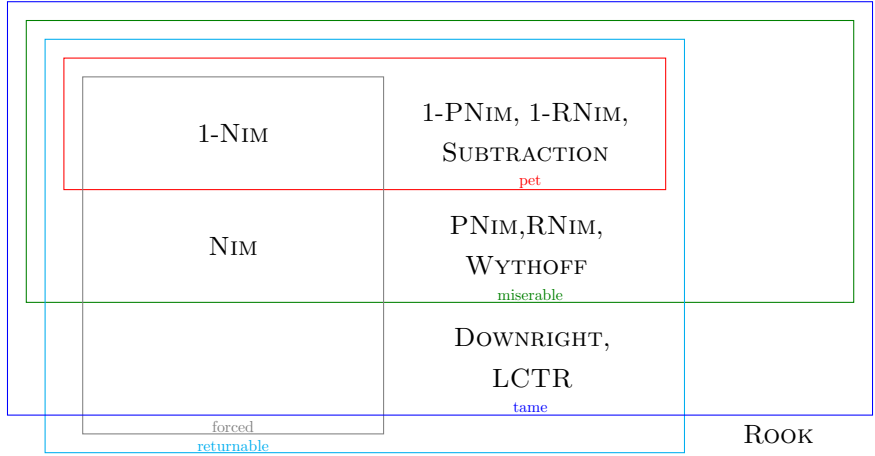

\centering
\class
\caption{\label{fig:CGclass}
The CGH classifications of $1$-\pnim{}, $1$-\rnim{}, \pnim{}, \rnim{}, and some known games; see \cite{Gottlieb2024LCTR,gottlieb2025impartialchessintegerpartitions,berlekamp2018winning,Wythoff}.
The domestic class is not relevant for this paper so we omit it. 
}
\end{figure}

\begin{cor}
    The games \pnim{} and \rnim{} are miserable, but not pet.
    They are also returnable, but not forced.
\end{cor}
\begin{proof}
    The property of being both miserable and returnable is closed under disjunctive sum, thus \pnim{} and \rnim{} are miserable and returnable; see \cite[Theorem 11, Proposition 2]{GURVICH201854}. 
    \pnim{} and \rnim{} are not pet, because \nim{} is not pet and a single row $\br{n}$ of \pnim{} is equivalent to $\rs{n,1,\ldots,1}$ in \rnim{} which is equivalent to a single pile of \nim{} with $n$ tokens.
\end{proof}
The location of $1$-\pnim{}, $1$-\rnim{}, \pnim{}, and \rnim{} in the CGH classification scheme is shown in \cref{fig:CGclass}.


\subsection{Computing mis\`ere-Grundy values}

Given tame games of known Sprague-Grundy values, we can compute the mis\`ere-Grundy value of their disjunctive sum:
\begin{thm}[Conway {\cite[p. 178]{Con76}}, Gurvich et al. \cite{GURVICH201854}]\label{thm:tame}
    Let $p_1,\ldots, p_n$ be tame games. Then their disjunctive sum is a
    \begin{itemize}
    \item $(1,0)$-position if and only if it has and odd number of $1$-positions among $p_i$ and the rest are $0$-positions.
    \item $(0,1)$-position if and only if it has and even number of $1$-positions among $p_i$ and the rest are $0$-positions.
    \item $(k,k)$-position for some $k\geq0$ otherwise, where $k$ is the Nim-sum of the Sprague-Grundy values of the $p_i$'s.
    \end{itemize}
\end{thm}

Using the above theorem alongside \cref{cor:rnimformula} we can efficiently compute normal and mis\`ere Grundy values for \rnim{}; see \cref{tab:calctame} for examples.

\begin{table}[!h]
    \centering
    $\begin{array}{c|c|c}
        \text{Position in \rnim{}} & \sgrnim^{\phantom{}} & \sgrnim^- \\ \hline 
        \rs{2,2} + \rs{4,3,2} + \rs{1,1} & 1 \oplus 1 \oplus 1=1 & 0 \\
        \rs{1,0,2} + \rs{2,3,4} + \rs{4,4} & 0 \oplus 1\oplus 1=0 & 1 \\
        \rs{1,2,3} + \rs{2,2} & 4 \oplus 1=5 & 5 \\
        \rs{1,2,3} + \rs{3,2} & 4 \oplus 4=0 & 0 
    \end{array}$
    \caption{
    Computing $\sgrnim$- and $\sgrnim^-$-values via \cref{thm:tame}.}
    \label{tab:calctame}
\end{table}

\subsection{How to play mis\`ere 1-\pnim{}}\label{sec:how-to-play}

Recall that \pnim{} is pet (see \cref{prop:pet}) and that we have characterized all $1$-positions (see \cref{thm:partition_interval}) and $0$-positions (only the empty partition) in 1-\pnim{}. Thus $1$-positions in 1-\pnim{} are precisely the losing positions in mis\`ere 1-\pnim{}.
In particular, a player with a winning strategy simply needs to make a move leading to partition in 
\[\{\lambda : \br{r,r,r-1,r-2,\ldots,2}\leq \lambda \leq \br{r^r}\} \mbox{ for some } r\geq 2\}.\] 
The proof of \cref{thm:partition_interval}  shows how to find such a move.

\section{Conclusion} \label{sec:conclusion}

It is interesting to us that the analyses of \pnim{} and \rnim{} are connected to tools from different areas of mathematics, like XOR addition (not arising from the disjunctive sum of games) and Lucas's theorem.
If $n\ge 5$, then the Sprague-Grundy value of the rectangular partition $\br{n,n}$ is equal to the number of (undirected) Hamiltonian cycles in the $(n-2)$-Möbius ladder \cite{oeis_A103889}.

More generally, the matrix $\mathbb M =\left( \sgpnim(\br{c^r}) \right)_{r, c = 1}^{\infty}$ is related to Sierpinski triangle.
The matrix $\mathbb M$ also has the interesting cyclical property that $k$ appears in entry $(i, j)$ if and only if $j$ appears in entry $(k, i)$. 
Finally, reading $\mathbb M$ by antidiagonals, it is lexicographically first among all matrices of positive integers for which no row or column contains a repeated entry \cite{oeis_A280172}.

The games \pnim{} and \rnim{} reveal several patterns yet to be understood.
Recall our characterizations of $1$-positions in $1$-\pnim{}.
\thmpartitioninterval*
\begin{question}\label{q:sg2}
    Which partitions have Sprague-Grundy value (and thus also mis\`ere-Grundy value) $2$ in \pnim{}?
\end{question}
In \cref{app:sg2} we list all partitions with of order at most $15$ 
with $\sgpnim$ value $2$.
Also related to \cref{thm:partition_interval}, observe that removing squares from a rectangular partition ``under the main anti-diagonal'' does not affect its Sprague-Grundy value.
To our surprise, computational data suggest that a similar statement can be made about heavy rectangles (see \cref{sec:heavy}).
\begin{conj}\label{conj:heavychoppedrect}
    Let $a,b\in\ZZ$ be such that $\br{(a+1)^{b+1}}$ is heavy. 
    Then any $\lambda$ satisfying
    $\br{a+1,a,\dots, a-b+1}\le\lambda\le \br{(a+1)^{b+1}}$ is heavy.
\end{conj}
\noindent Here is another family of partitions that we conjecture to be heavy.
\begin{conj}\label{conj:shallowstair}
    If $i$, $s$, and $k$ are positive integers, then $\br{i + (k-1)s, \ldots, i+s, i}$ is heavy.
\end{conj}

\noindent
Heavy partitions seem to appear in regular patterns. 
\begin{question}\label{conj:heavy}
    Which partitions are heavy?
\end{question}
\noindent In \cref{sec:heavy} we list all heavy partitions of $n\le 8$, up to conjugation.

\subsection*{Acknowledgements}
This work is supported in part by internal grants from Rhodes College, by the  Slovenian Research and Innovation Agency (research program P1-0383 and research projects 
J1-3003,
J1-4008,
J5-4596, 
BI-US/22-24-164, 
BI-US/22-24-093, 
BI-US/24-26-018, 
and
N1-0370).

\appendix
\renewcommand{\thesection}{\Alph{section}}
\newpage
\section{Small partitions with $\sgpnim$ value $2$\label{app:sg2}} \label{sec:small}
All partitions $\lambda$ with $\sgpnim(\lambda)=2$ of  order at most $n = 26$, up to conjugation:\\[6mm]
    {\centering
    \begin{tabular}{>{\centering\arraybackslash}m{18mm}>{\centering\arraybackslash}m{18mm}>{\centering\arraybackslash}m{18mm}>{\centering\arraybackslash}m{18mm}>{\centering\arraybackslash}m{18mm}}   
\youngdiagram{1, 1}& 
\youngdiagram{3, 3, 3, 1}& 
\youngdiagram{3, 3, 3, 3}& 
\youngdiagram{5, 4, 4, 3, 1}& 
\youngdiagram{5, 5, 5, 3, 3, 1}\\[11mm]
\youngdiagram{5, 5, 5, 4, 3, 1}& 
\youngdiagram{5, 5, 5, 5, 3, 1}& 
\youngdiagram{5, 5, 5, 4, 4, 1}& 
\youngdiagram{5, 5, 5, 3, 3, 3}& 
\youngdiagram{5, 5, 5, 5, 4, 1}\\[11mm]
&
\youngdiagram{5, 5, 5, 5, 5, 1}& 
\youngdiagram{5, 5, 5, 5, 3, 3}& 
\youngdiagram{5, 5, 5, 4, 4, 3}& 
    \end{tabular}}
  %
\newpage
\section{Small heavy partitions under \pnim{}} \label{sec:heavy}
All heavy partitions under \pnim{} of order at most $8$, up to conjugation:\\[6mm]
   {\centering
    \begin{tabular}{>{\centering\arraybackslash}m{13mm}>{\centering\arraybackslash}m{13mm}>{\centering\arraybackslash}m{13mm}>{\centering\arraybackslash}m{13mm}>{\centering\arraybackslash}m{13mm}>{\centering\arraybackslash}m{13mm}}
\youngdiagram{1}& \youngdiagram{1, 1}& \youngdiagram{2, 1}& \youngdiagram{1, 1, 1}& \youngdiagram{2, 1, 1}& 
\youngdiagram{1, 1, 1, 1}\\[6mm]
\youngdiagram{3, 1, 1}& \youngdiagram{2, 2, 1}& \youngdiagram{2, 1, 1, 1}& \youngdiagram{1, 1, 1, 1, 1}&
\youngdiagram{3, 2, 1}& \youngdiagram{3, 1, 1, 1}\\[8mm] 
\youngdiagram{2, 2, 2}& \youngdiagram{2, 2, 1, 1}& \youngdiagram{2, 1, 1, 1, 1}&
\youngdiagram{1, 1, 1, 1, 1, 1}& \youngdiagram{4, 1, 1, 1}& \youngdiagram{3, 2, 2}\\[9mm]
\youngdiagram{3, 2, 1, 1}& \youngdiagram{3, 1, 1, 1, 1}&
\youngdiagram{2, 2, 2, 1}& \youngdiagram{2, 2, 1, 1, 1}& \youngdiagram{2, 1, 1, 1, 1, 1}& \youngdiagram{1, 1, 1, 1, 1, 1, 1}\\[12mm]
\youngdiagram{4, 2, 1, 1}&
\youngdiagram{4, 1, 1, 1, 1}& \youngdiagram{3, 3, 1, 1}& \youngdiagram{3, 2, 2, 1}& \youngdiagram{3, 2, 1, 1, 1}& \youngdiagram{3, 1, 1, 1, 1, 1}\\[11mm]
&
\youngdiagram{2, 2, 2, 1, 1}& \youngdiagram{2, 2, 1, 1, 1, 1}& \youngdiagram{2, 1, 1, 1, 1, 1, 1}& \youngdiagram{1, 1, 1, 1, 1, 1, 1, 1}&
    \end{tabular}}


\begin{thebibliography}{10}
	\expandafter\ifx\csname url\endcsname\relax
	\def\url#1{\texttt{#1}}\fi
	\expandafter\ifx\csname urlprefix\endcsname\relax\def\urlprefix{URL }\fi
	\expandafter\ifx\csname href\endcsname\relax
	\def\href#1#2{#2} \def\path#1{#1}\fi
	
	\bibitem{bouton1901}
	C.~L. Bouton, \href{http://www.jstor.org/stable/1967631}{Nim, a game with a
		complete mathematical theory}, Annals of Mathematics 3~(1/4) (1901) 35--39.
	\newline\urlprefix\url{http://www.jstor.org/stable/1967631}
	
	\bibitem{Spr35}
	R.~P. Sprague, \"{U}ber mathematische kampfspiele, Tohoku Mathematical Journal,
	First Series 41 (1935) 438--444.
	
	\bibitem{Gru39}
	P.~M. Grundy, Mathematics of games, Eureka 2 (1939) 6--8.
	
	\bibitem{sato1970maya}
	M.~Sato, On maya game (notes by h. enomoto), Sugaku no Ayumi 15~(1) (1970)
	73--84.
	
	\bibitem{irie2018p}
	Y.~Irie, p-saturations of welter’s game and the irreducible representations
	of symmetric groups, Journal of Algebraic Combinatorics 48 (2018) 247--287.
	
	\bibitem{abuku2023multiple}
	T.~Abuku, M.~Tada, A multiple hook removing game whose starting position is a
	rectangular {Y}oung diagram with unimodal numbering, Integers 23 (2023) Paper
	No. G1, 37.
	
	\bibitem{motegi2021gamepositionsmultiplehook}
	Y.~Motegi, \href{https://arxiv.org/abs/2112.14200}{Game positions of multiple
		hook removing game} (2021).
	\newblock \href {http://arxiv.org/abs/2112.14200} {\path{arXiv:2112.14200}}.
	\newline\urlprefix\url{https://arxiv.org/abs/2112.14200}
	
	\bibitem{Gottlieb2024LCTR}
	E.~Gottlieb, M.~Krnc, P.~Muršič, Sprague–{G}rundy values and complexity for
	{LCTR}, Discrete Applied Mathematics 346 (2024) 154--169.
	\newblock \href {https://doi.org/10.1016/j.dam.2023.11.036}
	{\path{doi:10.1016/j.dam.2023.11.036}}.
	
	\bibitem{Ilic2019}
	E.~Gottlieb, J.~Ili{\'c}, M.~Krnc, Some results on {LCTR}, an impartial game on
	partitions, Involve, a Journal of Mathematics 16~(3) (2023) 529--546.
	\newblock \href {https://doi.org/10.2140/involve.2023.16.529}
	{\path{doi:10.2140/involve.2023.16.529}}.
	
	\bibitem{basic2022some}
	I.~Ba{\v{s}}i{\'c}, E.~Gottlieb, M.~Krnc, Some observations on the {Column-Row}
	game, in: Proceedings of the 9th Student Computing Research Symposium
	(SCORES’23), 2022.
	\newblock \href {https://doi.org/10.26493/scores23}
	{\path{doi:10.26493/scores23}}.
	
	\bibitem{Basic_2023}
	I.~Bašić,
	\href{https://repozitorij.upr.si/IzpisGradiva.php?id=19707}{{Column-Row}
		game}, Master's thesis (2023).
	\newline\urlprefix\url{https://repozitorij.upr.si/IzpisGradiva.php?id=19707}
	
	\bibitem{impchess}
	E.~R. Berlekamp, {Impartial Chess},
	\href{https://math.berkeley.edu/~berlek/research_interests/games/ichess.html}{\texttt{https://math.berkeley.edu/{\~{}}berlek/}}
	(2017).
	
	\bibitem{gottlieb2025impartialchessintegerpartitions}
	E.~Gottlieb, M.~Krnc, P.~Muršič,
	\href{https://arxiv.org/abs/2501.14640}{Impartial chess on integer
		partitions} (2025).
	\newblock \href {http://arxiv.org/abs/2501.14640} {\path{arXiv:2501.14640}}.
	\newline\urlprefix\url{https://arxiv.org/abs/2501.14640}
	
	\bibitem{hannah2025thesis}
	H.~Meit, Two partizan games on integer partition, Master's thesis, Rhodes
	College (2025).
	
	\bibitem{dyson1944some}
	F.~J. Dyson, Some guesses in the theory of partitions, Eureka (Cambridge)
	8~(10) (1944) 10--15.
	
	\bibitem{andrews1998theory}
	G.~E. Andrews, The theory of partitions, Cambridge Mathematical Library,
	Cambridge University Press, Cambridge, 1998, reprint of the 1976 original.
	
	\bibitem{Sie13}
	A.~N. Siegel, Combinatorial game theory, Vol. 146 of Graduate Studies in
	Mathematics, American Mathematical Society, Providence, RI, 2013.
	\newblock \href {https://doi.org/10.1090/gsm/146} {\path{doi:10.1090/gsm/146}}.
	
	\bibitem{Con76}
	J.~H. Conway, On numbers and games, 2nd Edition, A K Peters, Ltd., Natick, MA,
	2001.
	
	\bibitem{berlekamp2018winning}
	E.~Berlekamp, J.~Conway, R.~Guy,
	\href{https://books.google.si/books?id=-0laDwAAQBAJ}{Winning Ways for Your
		Mathematical Plays: Volume 1-4}, AK Peters/CRC Recreational Mathematics
	Series, CRC Press, 2018.
	\newline\urlprefix\url{https://books.google.si/books?id=-0laDwAAQBAJ}
	
	\bibitem{lucas}
	N.~J. Fine, \href{http://www.jstor.org/stable/2304500}{Binomial coefficients
		modulo a prime}, The American Mathematical Monthly 54~(10) (1947) 589--592.
	\newline\urlprefix\url{http://www.jstor.org/stable/2304500}
	
	\bibitem{rowland2011numbernonzerobinomialcoefficients}
	E.~Rowland, \href{https://arxiv.org/abs/1001.1783}{The number of nonzero
		binomial coefficients modulo $p^\alpha$} (2011).
	\newblock \href {http://arxiv.org/abs/1001.1783} {\path{arXiv:1001.1783}}.
	\newline\urlprefix\url{https://arxiv.org/abs/1001.1783}
	
	\bibitem{GURVICH201854}
	V.~A. Gurvich, N.~B. Ho, On tame, pet, domestic, and miserable impartial games,
	Discrete Applied Mathematics 243 (2018) 54--72.
	\newblock \href {https://doi.org/10.1016/j.dam.2017.12.006}
	{\path{doi:10.1016/j.dam.2017.12.006}}.
	
	\bibitem{Wythoff}
	W.~A. Wythoff, A modification of the game of {Nim}, Nieuw Archiefvoor Wiskunde
	(1907-1908) 199--202.
	
	\bibitem{oeis_A103889}
	{OEIS Foundation Inc.}, The on-line encyclopedia of integer sequences,
	\url{https://oeis.org/A103889} (2005).
	
	\bibitem{oeis_A280172}
	{OEIS Foundation Inc.}, The on-line encyclopedia of integer sequences,
	\url{https://oeis.org/A280172} (2016).
	
\end{thebibliography}
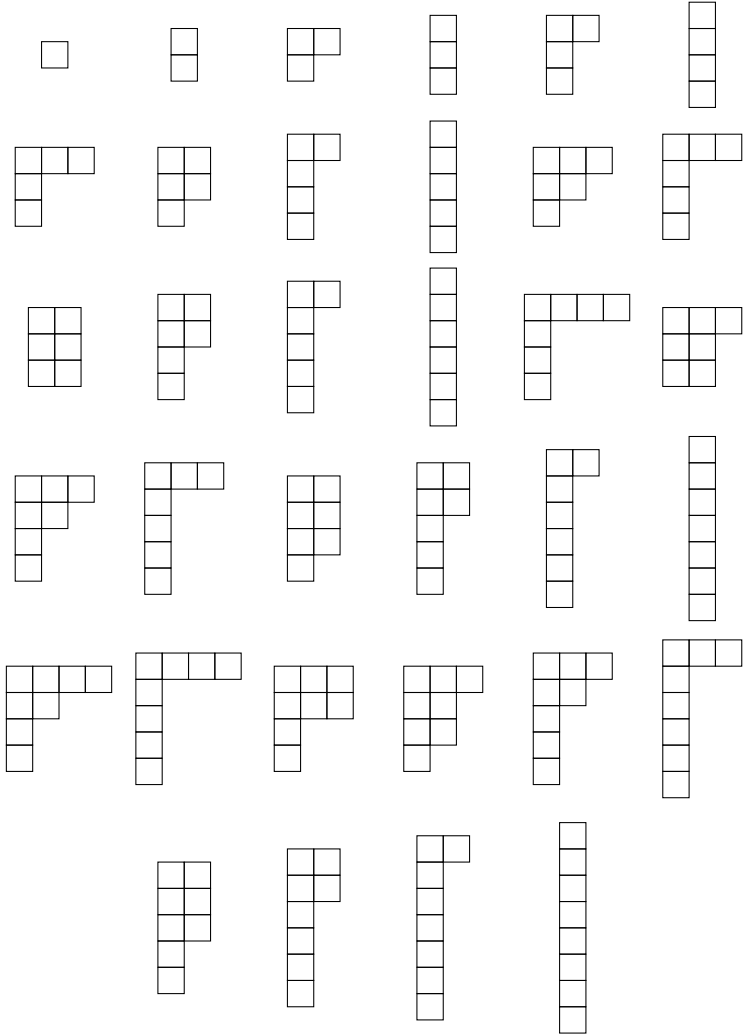
\end{document}